\documentclass[11pt,english,online]{smfart}

\usepackage[T1]{fontenc}
\usepackage[francais]{babel}

\usepackage{amssymb,url,xspace,smfthm,amsmath,amsfonts,enumitem}
\usepackage{tikz}
\usetikzlibrary{matrix,positioning}

\DeclareFontFamily{U}{wncy}{}
\DeclareFontShape{U}{wncy}{m}{n}{
<5>wncyr5
<6>wncyr6
<7>wncyr7
<8>wncyr8
<9>wncyr9
<10>wncyr10
<11>wncyr10
<12>wncyr6
<14>wncyr7
<17>wncyr8
<20>wncyr10
<25>wncyr10}{}
\DeclareMathAlphabet{\cyr}{U}{wncy}{m}{n}

\newcommand{\BibTeX}{{\scshape Bib}\kern-.08em\TeX}
\newcommand{\T}{\S\kern .15em\relax }
\newcommand{\AMS}{$\mathcal{A}$\kern-.1667em\lower.5ex\hbox
        {$\mathcal{M}$}\kern-.125em$\mathcal{S}$}

\usepackage{tikz}
\usetikzlibrary{matrix}

\title[Log good reduction, monodromy and the rational volume]{Logarithmic good reduction, \\ monodromy and the rational volume}
\date {October 2016}
\author{Arne Smeets}

\address{Radboud Universiteit Nijmegen, IMAPP, Heyendaalseweg 135, 6525 AJ Nijmegen, The Netherlands \emph{and} University of Leuven, Departement Wiskunde, Celestijnenlaan 200B, 3001 Heverlee, Belgium} 
\email{arnesmeets@gmail.com}
\urladdr{}
\begin{document}
\maketitle
\begin{abstract} Let $R$ be a strictly local ring complete for a discrete valuation, with fraction field $K$ and residue field of characteristic $p > 0$. Let $X$ be a smooth, proper variety over $K$. Nicaise conjectured that the rational volume of $X$ is equal to the trace of the tame monodromy operator on $\ell$-adic cohomology if $X$ is cohomologically tame. He proved this equality if $X$ is a curve. We study his conjecture from the point of view of logarithmic geometry, and prove it for a class of varieties in any dimension: those having logarithmic good reduction. 
\end{abstract}
\tableofcontents
\section{Introduction}
Let $R$ be a strictly local ring complete for a discrete valuation, with fraction field $K$. Assume that the residue field $k$ of $R$ is algebraically closed, of characteristic $p > 0$. Fix a prime number $\ell \neq p$. Let $K^s$ be a separable closure of $K$, and let $K^t$ be the tame closure of $K$ inside $K^s$. Let $P = \mathrm{Gal}(K^s/K^t)$ be the wild inertia group.

Let $\varphi$ be a topological generator of the tame inertia group $I^t = \mathrm{Gal}(K^t/K)$, which is procyclic; we will refer to $\varphi$ as the \emph{tame monodromy operator}.

Let $X$ be a proper, smooth $K$-variety. One cannot simply ``count'' the number of rational points on $X$, but one can use other measures for the number of rational points on $X$. One such measure, the \emph{rational volume}, can be constructed using a so-called \emph{weak N\'eron model} of $X$, of which we recall the definition:

\begin{defi} \label{defi:weakneron} A \emph{weak N\'eron model} for $X$ is a smooth, separated $R$-scheme of finite type $\mathcal{X}$, endowed with an isomorphism $\mathcal{X} \times_R K \cong X$, such that the natural map $\mathcal{X}(R) \to \mathcal{X}(K) = X(K)$ is a bijection.  \end{defi} Such a model always exists: one can simply take the smooth locus of a N\'eron smoothening of any proper $R$-model of $X$, cfr. \cite[Theorem 3.1.3]{blr}.  \begin{defi} \label{defi:rationalvolume} The \emph{rational volume} $\mathrm{s}(X)$ of $X$ is the $\ell$-adic Euler characteristic of the special fibre of a weak N\'eron model for $X$. \end{defi}

That the integer $\mathrm{s}(X)$ does not depend on the choice of a weak N\'eron model is a deep result; the only known proof for this fact uses the theory of motivic integration. For more details, we refer to the foundational paper of Loeser--Sebag \cite{LS}, and subsequent work by Nicaise--Sebag \cite{NicaiseSebag}, Nicaise \cite{nic1} and Esnault--Nicaise \cite[\S 3]{EN} on the \emph{motivic Serre invariant}. This is the class in $K_0^R(\mathrm{Var}_k)/(\mathbb{L} - 1)$ of the special fibre of a weak N\'eron model; here $K_0^R(\mathrm{Var}_k)$ denotes a ``modified'' Grothendieck ring of varieties over $k$ (see \cite[\S 2]{EN} for the precise definition), and $\mathbb{L}$ is the class of $\mathbb{A}^1_k$ in this ring. The motivic Serre invariant is already independent of the choice of such a model (see e.g. \cite[Theorem 3.6]{EN}), and hence so is its realization $\mathrm{s}(X)$.

One has $\mathrm{s}(X) = 0$ if $X(K) = \emptyset$, since then $X$ (viewed as an $R$-scheme) is a weak N\'eron model for itself. Nicaise asked in \cite[\S 6]{nic1} whether the rational volume has a cohomological interpretation similar to the Grothendieck--Lefschetz formula: if the variety $X$ is cohomologically tame, i.e. if the wild inertia subgroup $P$ acts trivially on the \'etale cohomology groups $H^i(X \times_K K^s,\mathbf{Q}_\ell)$, and if moreover $X(K^t) \neq \emptyset$, do we have the equality \begin{equation} \label{traceformulaz} \mathrm{s}(X) = \sum_{i \geq 0} (-1)^i \mathrm{Tr}\left(\varphi \mid H^i\left(X \times_K K^t,\mathbf{Q}_\ell\right)\right)\ ? \end{equation}

The left-hand side of this formula is defined from the special fibre of an integral model of $X$, whereas the right-hand side (which is a priori an element of $\mathbf{Q}_\ell$) involves only the generic fibre. In their groundbreaking paper \cite{NicaiseSebag}, Nicaise--Sebag proved a formula of this type for non-archimedean analytic spaces; their formula is an arithmetic analogue of the one obtained by Denef--Loeser in \cite{denefloeser}. 

Subsequently, Nicaise proved the equality (\ref{traceformulaz}) in equal characteristic zero \cite[\S 6]{nic1}. However, the situation becomes much more subtle if the residual characteristic is positive: Nicaise proved the equality (\ref{traceformulaz}) if $X$ is a curve \cite[\S 7]{nic1}, and Halle--Nicaise handled the case of a semi-abelian variety \cite[Chapter 8]{halnic}. 

In this paper, we study the conjectural formula (\ref{traceformulaz}) in the framework of logarithmic geometry in order to prove it for a large class of varieties in any dimension, those with \emph{log good reduction}. Our main result can be stated as follows:

\begin{theo} \label{maintheorem} Let $X$ be a proper, smooth $K$-variety. Assume that there exists a flat, proper $R$-model $\mathcal{X}$ of $X$ such that $\mathcal{X}^\dagger$ is log smooth over $R^\dagger$. Then (\ref{traceformulaz}) holds, i.e. $$\label{traceformula} \mathrm{s}(X) = \sum_{i \geq 0} (-1)^i \mathrm{Tr}\left(\varphi \mid H^i\left(X \times_K K^t,\mathbf{Q}_\ell\right)\right).$$  \end{theo} Here the log scheme $\mathcal{X}^\dagger$ stands for the model $\mathcal{X}$ equipped with the natural log structure, i.e. the divisorial log structure induced by the special fibre $\mathcal{X}_s$, and $R^\dagger$ is the ``log ring'' given by the inclusion $R \setminus \{0\} \hookrightarrow R$.  We want to stress that the underlying scheme $\mathcal{X}$ may very well be singular, even if $\mathcal{X}^\dagger$ is log smooth.

By work of Nakayama \cite[Corollary 0.1.1]{nakayama}, varieties with log good reduction are automatically cohomologically tame, i.e. they satisfy the major hypothesis in Nicaise's conjecture. Nicaise also assumed the existence of a $K^t$-point in the statement of his conjecture, but we will not need this assumption: our method works for all varieties with log good reduction, with or without a $K^t$-point. Nevertheless, it would be interesting to know whether the log good reduction hypothesis implies the existence of a $K^t$-point in general, and as far as we know, this question is open.

Conversely, it is known in some cases that the cohomological tameness assumption implies the existence of a proper, log smooth model. Indeed, if $X$ is a cohomologically tame curve of genus at least $2$, then $X$ has log good reduction, by work of T. Saito \cite{saito,saito2} and Stix \cite[Theorem 1.2]{stix}. Hence we recover Nicaise's results for curves \cite[\S 7]{nic1}. Moreover, in the recent preprint \cite{BS}, A. Bellardini and the author managed to prove the existence of a proper log smooth model for any cohomologically tame abelian variety. Hence our Theorem \ref{maintheorem} also implies the result of Halle--Nicaise on the validity of the trace formula (\ref{traceformulaz}) for abelian varieties \cite[Chapter 8]{halnic}.

Our proof is very much inspired by the strategy used by Nicaise--Sebag in \cite{NicaiseSebag} and Nicaise in \cite{nic1}; however, the technicalities become much more complicated, and a new idea will be needed to finish the proof. Let $\mathcal{X}$ be a model of $X$ over $R$ such that the log scheme $\mathcal{X}^\dag$ is \emph{log regular} in the sense of Kato \cite{kato}. One can associate with $\mathcal{X}^\dagger$ its \emph{fan} $F(\mathcal{X})$, which is a finite monoidal space. There exists a morphism of monoidal spaces $\pi: (\mathcal{X},\mathcal{M}_X^\sharp) \to F(\mathcal{X})$ which we call the \emph{characteristic morphism}; here $\mathcal{M}_X$ is the sheaf of monoids defining the log structure on $\mathcal{X}^\dagger$, and $\mathcal{M}_X^\sharp = \mathcal{M}_X/\mathcal{O}_X^\times$. With each $x \in F(\mathcal{X})$, one then associates the locally closed subset $\pi^{-1}(x) = U(x)$ of $\mathcal{X}$. These sets give the so-called \emph{logarithmic stratification} $(U(x))_{x \in F(\mathcal{X})}$ of $\mathcal{X}$. 

Inspired by the computations on strict normal crossings models in \cite{NicaiseSebag} and \cite{nic2}, we will give explicit formulae for both the trace of the tame monodromy operator and the rational volume in terms of the logarithmic stratification. The trace of the monodromy operator can be computed using the logarithmic description of the sheaves of tame nearby cycles, first given by Nakayama in the log smooth case  \cite{nakayama} and later generalized by Vidal \cite{vidal}. We obtain \begin{equation} \label{tracemon} \sum_{i \geq 0} (-1)^i\mathrm{Tr}\left(\varphi^d \mid H^i\left(X \times_K K^t, \mathbf{Q}_\ell\right)\right) = \sum_{\substack{x \in F(\mathcal{X})^{(1)}  \\ s(x)' \mid d}}  s(x)' \chi(U(x)),\end{equation} where $F(X)^{(1)}$ denotes the set of height $1$ points in the fan $F(\mathcal{X})$, $s(x) \in \mathbf{N}$ is an integer which can be read off from the log structure and $s(x)'$ denotes the biggest prime-to-$p$ divisor of $s(x)$. The rational volume can be computed using the formalism of log blow-ups developed by Kato \cite[\S 10]{kato} (see also Niziol's work \cite{niziol}). This gives \begin{equation} \label{ratvol} \mathrm{s}(X) = \sum_{\substack{x \in F(\mathcal{X})^{(1)} \\ s(x) = 1}} \chi(U(x)). \end{equation} Using the equalities (\ref{tracemon}) and (\ref{ratvol}), one finds \begin{equation} \label{epsilon} \sum_{i \geq 0} (-1)^i \mathrm{Tr}(\varphi \mid H^i(X \times_K K^t,\mathbf{Q}_\ell)) - \mathrm{s}(X) = \sum_{\substack{x \in F(\mathcal{X})^{(1)} \\ s(x) = p^r,\ r \geq 1}} \chi(U(x)). \end{equation}

The result then follows from the fact that each term $\chi(U(x))$ in the right-hand side of the equality (\ref{epsilon}) vanishes. In fact, we prove the more general result that whenever $x \in F(\mathcal{X})^{(1)}$ is a point for which $p$ divides $s(x)$, then $\chi(U(x)) = 0$. This is perhaps the principal novelty of this paper. It is only at this point that the logarithmic smoothness assumption needs to be invoked; the previous steps only require the (significantly weaker) assumption that there exists a log regular model. The key observation involves the sheaves of logarithmic $1$-forms on the relevant strata.

The paper is organized as follows. For the convenience of the reader, we will recall a few important notions from logarithmic geometry in \S 2. In \S 3, we prove some technical results on sheaves of tame nearby cycles in the logarithmic setting; we use these to compute the tame monodromy zeta function and to prove the equality (\ref{tracemon}). In \S 4, we use resolution of toric singularities \`a la Kato--Niziol to prove the formula (\ref{ratvol}). In \S 5, we obtain the crucial new ingredient needed to finish the proof. 

\subsection*{Acknowledgements} Most of the results in this paper were contained in the author's PhD thesis, written in Leuven and Orsay. I warmly thank Johannes Nicaise for proposing me to work on this beautiful subject (which owes a lot to his ideas) and for his help and encouragement. Special thanks are due to Takeshi Saito for a beautiful suggestion which led me to the argument given in \S 5. Moreover, I would like to thank Dan Abramovich, Alberto Bellardini, Emmanuel Bultot, H\'el\`ene Esnault, Jakob Stix, Wim Veys and Olivier Wittenberg for useful suggestions and discussions. The author was supported by FWO Vlaanderen and the European Research Council's FP7 programme under ERC Grant Agreement nr. 615722 (MOTMELSUM).

\section{Preliminaries on logarithmic geometry}

For basic definitions and properties of monoids and logarithmic schemes, we refer to Kato's foundational text \cite{katof} and to the detailed treatments given by Ogus \cite{ogus} and Gabber--Ramero \cite[Chapter 9]{GR}. In this section, we briefly recall a few notions which will play an important role in this paper.

\subsection{Notation} We will denote a log scheme by $(X,\mathcal{M}_X)$, where $X$ is the underlying scheme and $\mathcal{M}_X$ is the sheaf of monoids defining the log structure on $X$, endowed with a morphism of sheaves $\alpha_X: \mathcal{M}_X \to \mathcal{O}_X$.  One can define log structures both on Zariski sites and \'etale sites; to keep things simple, we will work with Zariski log structures throughout, since these are sufficient for our purposes. 

Given a monoid $P$, we write $P^\sharp$ for the associated sharp monoid $P/P^\times$, $P^{\mathrm{gp}}$ for the group envelope of $P$ and $P^{\mathrm{sat}}$ for the saturation of $P$ in $P^{\mathrm{gp}}$.

\subsection{Divisorial log structures} Let $X$ be a locally Noetherian scheme and let $D \hookrightarrow X$ be a divisor on $X$. Let $j: U \hookrightarrow X$ be the corresponding open immersion, where $U = X \setminus D$. Then $$\mathcal{M}_X = \mathcal{O}_X \cap j_\star \mathcal{O}_U^\times \hookrightarrow \mathcal{O}_X$$ defines a log structure on $X$, the \emph{divisorial log structure} induced by $D$. 

A special case of this construction is the following. Let $S$ be a trait, i.e. $S = \mathrm{Spec}\,R$ where $R$ is a discrete valuation ring. Let $\pi$ be a uniformizer. Then $R  \setminus \{0\} \hookrightarrow R$ defines a log structure on $R$, \emph{the standard log structure}, which is exactly the divisorial log structure induced by the divisor given by $\pi = 0$.  We will denote the corresponding log scheme by $S^\dagger$. If $X$ is a flat $R$-scheme, then the special fibre $X_s$ is a divisor; we denote by $X^\dagger$ the log scheme obtained by equipping $X$ with the divisorial log structure induced by $X_s$. This yields a well-defined morphism $X^\dagger \to S^\dagger$ of log schemes. (Sometimes we will simply write $R^\dagger$ instead of $S^\dagger$.)

\subsection{Fibre products}
Fibre products in the category $\mathsf{Log}^{\mathsf{fs}}$ of fs log schemes will be denoted by $\times^{\mathrm{fs}}$. If $S = \mathrm{Spec}\,R$ is a trait and $\pi \in R$ a uniformizer, we define for $d \in \mathbf{Z}_{>0}$ the scheme $S(d) = \mathrm{Spec}\,R[T]/(T^d - \pi)$. Then $u_d: \mathbf{N} \hookrightarrow \frac{1}{d}\mathbf{N}$ gives a chart for the morphism $S(d)^\dagger \to S^\dagger$. Given any log scheme $X^\dagger$ over $S^\dagger$, define \begin{equation} \label{X(d)} X(d)^\dagger = X^\dagger \times_{S^\dagger}^{\mathsf{fs}} S(d)^\dagger. \end{equation} Recall that fibre products in $\mathsf{Log}^{\mathsf{fs}}$ do not commute with the forgetful functor from $\mathsf{Log}^{\mathsf{fs}}$ to the category of schemes. A simple example illustrating this phenomenon is the following: the underlying scheme of $S(d)^\dagger \times_{S^\dagger}^{\mathsf{fs}} S(d)^\dagger$ is $S(d) \times \mathrm{Spec}\,\mathbf{Z}[\mathbf{Z}/d\mathbf{Z}],$ which is (geometrically) a disjoint union of $d$ copies of $S(d)$; on the other hand, the fibre product of schemes $S(d) \times_S S(d)$ is not normal.

\subsection{Log regularity and log smoothness} Kato introduced the notion of \emph{log regularity} in \cite[\S 2]{kato}. We will recall its definition for the convenience of the reader.
\begin{defi} \label{logregular}
Let $(X,\mathcal{M}_X)$ be an fs log scheme. Given any point $x$ of $X$, let $I(x,\mathcal{M}_X)$ be the ideal of $\mathcal{O}_{X,x}$ generated by  $\mathcal{M}_{X,x}^+ = \mathcal{M}_{X,x} \setminus \mathcal{O}_{X,x}^{\times}$. Let $C_{X,x}$ the closed subscheme of $\mathrm{Spec}\,\mathcal{O}_{X,x}$ defined by $I(x,\mathcal{M}_X)$. Then $(X,\mathcal{M}_X)$ is \emph{log regular} at $x$ if $C_{X,x}$ is regular (as a scheme) and $$\dim \mathcal{O}_{X,x} = \dim C_{X,x} + \text{rank}_{\mathbf{Z}} (\mathcal{M}_{X,x}^\sharp)^\text{gp}.$$ A \emph{log regular scheme} is a log scheme which is everywhere log regular. Whenever we say/assume that a log scheme is log regular, this implies that the log structure is fs.
\end{defi}

Let us also recall the notion of a log smooth morphism of log schemes. Developing this notion was in some sense the main motivation for the theory of logarithmic geometry; it allows one to treat certain singular objects as if they were smooth.

\begin{defi} \label{logsmooth} Let $f: (X,\mathcal{M}_X) \to (Y,\mathcal{M}_Y)$ be a morphism of fs log schemes. Then $f$ is \emph{log smooth} (resp. \emph{log \'etale}) if \'etale locally on $X$ and $Y$, there exists a chart for $f$, given by maps $P_X \to \mathcal{O}_X$, $Q_Y \to \mathcal{O}_Y$ and $u: Q \to P$, such that the kernel and the torsion part of the cokernel (resp. the kernel and cokernel) of $u^\text{gp}: Q^{\text{gp}} \to P^{\text{gp}}$ are finite groups of order invertible on $X$, and the map $X \to Y \times_{\mathrm{Spec}\, \mathbf{Z}[Q]} \mathrm{Spec}\,\mathbf{Z}[P]$ is classically smooth (at the level of the underlying schemes). \end{defi}

Recall that a log smooth log scheme over a log regular scheme is again log regular, and that log regular schemes are normal and Cohen-Macaulay.

\subsection{The logarithmic stratification} The \emph{characteristic sheaf} of a log scheme $(X,\mathcal{M}_X)$ is the quotient $\mathcal{M}_X^\sharp = \mathcal{M}_X/\mathcal{O}_X^\times.$ With a log regular scheme, one can associate a finite monoidal space, the \emph{fan} $F(X)$, as follows: as a set, $$F(X) = \left\{x \in X \mid I(x,\mathcal{M}_X) = \mathfrak{m}_x\right\},$$ where $\mathfrak{m}_x$ denotes the maximal ideal in $\mathcal{O}_{X,x}$. Endow the set $F(X)$ with the topology induced by the topology on $X$ and with the inverse image of the sheaf $\mathcal{M}_X^\sharp$. This is a fan, as proved by Kato in \cite[Proposition 10.1]{kato}. If $X$ is quasi-compact, then $F(X)$ is easily seen to be a \emph{finite} monoidal space. 

Given a log scheme $(X,\mathcal{M}_X)$, one often considers the \emph{rank stratification of $X$}, constructed starting from the characteristic sheaf: denote $$X_i = \{x \in X: \mathrm{rank}_{\mathbf{Z}}\,(\mathcal{M}_{X,x}^\sharp)^\mathrm{gp} \geq i\}.$$ The sets $X_i \setminus X_{i - 1}$ are locally closed and give the rank stratification of $X$. In this paper, we use a finer stratification obtained from the fan of a log regular scheme $(X,\mathcal{M}_X)$. In \cite[\S 10.2]{kato}, Kato defines the \emph{characteristic morphism}: let $\pi: (X,\mathcal{M}_X^\sharp) \to F(X)$ map a point $x \in X$ to the point $\pi(x)$ of $F(X) \subseteq X$ which corresponds to the prime ideal $I(x,\mathcal{M}_X)$ of $\mathcal{O}_{X,x}$. This is an open morphism of monoidal spaces. Given $x \in F(X)$, let $U(x) := \pi^{-1}(x)$: this is a locally closed subset of $X$. This gives the  \emph{logarithmic stratification} $(U(x))_{x \in F(X)}$ of $X$. 

When endowed with the reduced scheme structure, each stratum is irreducible and regular \cite[Corollary 9.5.52(iii)]{GR}. For each $x \in F(X)$, we will denote by $\overline{U(x)}$ the Zariski closure of the stratum $U(x)$ inside $X$, again equipped with the reduced scheme structure. The \emph{height} $h(x)$ of a point $x$ in $F(X)$ is defined as the dimension of the monoid $\mathcal{M}_{F(X),x}$. The set of points of $F(X)$ of height equal to $i$ will be denoted by $F(X)^{(i)}$. We have the equality of locally closed, reduced subschemes $$X_i \setminus X_{i - 1} = \bigsqcup_{x \in F(X)^{(i)}} U(x).$$

\section{Nearby cycles and the monodromy zeta function} \label{sec:nearbycycles}

The main goal of this section is to calculate the tame monodromy zeta function of a smooth, proper $K$-variety $X$, starting from a proper $R$-model $\mathcal{X}$ such that $\mathcal{X}^\dagger$ is log regular. To do so, we use the ``logarithmic'' description of the sheaves of tame nearby cycles, first given by Nakayama \cite{nakayama} and later generalized by Vidal \cite{vidal}; along the way, we prove some technical results on nearby cycles.

Let us first introduce some easy terminology:

\begin{defi} A morphism of log schemes $f: (X,\mathcal{M}_X) \to (Y,\mathcal{M}_Y)$ is \emph{vertical} if for every $x \in X$, the induced homomorphism $\mathcal{M}_{Y,f(x)} \to \mathcal{M}_{X,x}$ is vertical; this means that the image of $\mathcal{M}_{Y,f(x)}$ is not contained in any proper face of $\mathcal{M}_{X,x}$. \end{defi}

A log scheme $(\mathcal{X},\mathcal{M}_\mathcal{X})$ is vertical over $R^\dagger$ if and only if $X = \mathcal{X} \times_R K$ is precisely the locus of triviality of the log structure. In particular, if $\mathcal{X}$ is a flat $R$-scheme, then $\mathcal{X}^\dagger$ is vertical over $R^\dagger$. We will then simply say that ``the log structure is vertical''.

\begin{defi} \label{saturationindex} Let $(\mathcal{X},\mathcal{M}_{\mathcal{X}})$ be a vertical log scheme over $R^\dagger$ which is log regular. Let $x \in F(\mathcal{X})^{(1)}$ be a height $1$ point in the fan of $\mathcal{X}$. Let $\mathbf{N} \to P$ be a chart for the log structure around $x$, where $P$ is an fs monoid. Then $x$ corresponds to a height $1$ prime ideal $\mathfrak{p}$ in $P$. Now $\mathbf{N} \to P$ induces a homomorphism $\mathbf{N} \to P/(P \setminus \mathfrak{p}) \cong \mathbf{N}$, which is multiplication by a positive integer $s(x)$. We will call this integer the \emph{saturation index} of $x \in F(\mathcal{X})^{(1)}$. The saturation index of $(\mathcal{X},\mathcal{M}_\mathcal{X})$ over $R^\dagger$ is the least common multiple of the integers $s(x)$ for $x \in F(\mathcal{X})^{(1)}$. \end{defi}

Let us now recall the definition of the sheaves of nearby cycles. Let $S = \mathrm{Spec}\,R$, let $\eta$ be its generic point and let $s$ be its closed point. Let $\eta^t$ be a tame closure of $\eta$ and let $S^t$ be the normalization of $S$ inside $\eta^t$. Let $\Lambda$ be $\mathbf{Z}/n\mathbf{Z}$, where $n$ is an integer prime to $p$, or one of $\mathbf{Z}_\ell$ and $\mathbf{Q}_\ell$. Let $\mathcal{X}$ be a scheme over $S$ and write $X = \mathcal{X}_\eta$. Consider the Cartesian squares
\begin{center}
\begin{tikzpicture}[auto]
\node (L1) {$\mathcal{X}_{s}$};
\node (L2) [below= 1.2cm of L1] {$s$};
\node (M1) [right= 1.2cm  of L1] {$\mathcal{X}^t$};
\node (M2) at (M1 |- L2) {$S^t$};
\node (R1) [right= 1.2cm  of M1] {$X^t$};
\node (R2) at (R1 |- M2) {$\eta^t$};

\draw[->] (L1) to node {\footnotesize $i^t$} (M1);
\draw[<-] (M1) to node {\footnotesize $j^t$} (R1);
\draw[->] (L2) to node {} (M2);
\draw[<-] (M2) to node {} (R2);
\draw[->] (L1) to node {} (L2);
\draw[->] (M1) to node {} (M2);
\draw[->] (R1) to node {} (R2);
\end{tikzpicture}
 \end{center}

Given $\mathcal{F}$ in $D^+(X,\Lambda)$, the derived category corresponding to bounded-below complexes of sheaves of $\Lambda$-modules on $X$, the corresponding complex of \emph{tame nearby cycles} is $$R\psi^t_\mathcal{X}(\mathcal{F}) := (i^t)^\star R(j^t)_\star(\mathcal{F}|_{X^t}) \in D^+(\mathcal{X}_s,\Lambda).$$  

Now let $X$ be any smooth, proper variety over $K$ and assume that there exists a proper, flat model $\mathcal{X}$ of $X$ over $R$ such that $\mathcal{X}^\dagger$ is \emph{log regular}. We will analyze the behaviour of the complex $R\psi^t_{\mathcal{X}}(\Lambda)$ on the logarithmic strata defined previously. We start with the following local computation of the trace of the tame monodromy operator $\varphi$:

\begin{lemm} \label{localcomp} With the above notation and assumptions, fix a point $x \in F(\mathcal{X})$. Take $z \in U(x)$ and let $\overline{z}$ be any geometric point lying over $z$. Then \begin{alignat*}{3}
  \mathrm{Tr}\left(\varphi^d \mid R\psi^t_\mathcal{X}(\Lambda)_{\overline{z}}\right) &=  s(x)' & &\ \ \text{if }x \in F(\mathcal{X})^{(1)} \text{ and } s(x)' \mid d, \\   \mathrm{Tr}\left(\varphi^d \mid R\psi^t_\mathcal{X}(\Lambda)_{\overline{z}}\right) & = 0 & &\ \ \text{else.} 
\end{alignat*} Here $s(x)'$ denotes the biggest prime-to-$p$ divisor of $s(x)$. \end{lemm}
\begin{proof} Nakayama gave a ``logarithmic'' description of the complex of nearby cycles in the case of log good reduction in \cite[Theorem 3.5]{nakayama}. This result was generalized by Vidal in \cite[\S 1.4]{vidal}; recall that the complex of nearby cycles is automatically tame in the case of log good reduction \cite[Theorem 3.2]{nakayama}. Their result is the following: let $\mathcal{C}$ be the locally constant sheaf  of finitely generated abelian groups on $\mathcal{X}_s$ given by $$\mathcal{C} = \text{coker}\left(f^\star(\mathcal{M}^{\sharp}_s)^{\text{gp}} \stackrel{\alpha}{\longrightarrow} (\mathcal{M}_{\mathcal{X}_s}^\sharp)^\text{gp}\right) \text{modulo torsion}.$$ For every integer $q \geq 1$, there is a natural, Galois equivariant isomorphism \begin{equation} R^0\psi^t_\mathcal{X}(\Lambda) \otimes \bigwedge^q \left(\mathcal{C}|_{\mathcal{X}_{\overline{s}}} \otimes \Lambda(-1)\right) \longrightarrow R^q \psi^t_\mathcal{X}(\Lambda) \label{wedgeproduct} \end{equation} since the log structure is vertical. The stalk of $R^0\psi^t_\mathcal{X}(\Lambda)$ at a classical geometric point $\overline{z}$ lying over $\overline{s}$ is non-canonically isomorphic to $\Lambda[E_{\overline{z}}]$, where $E_{\overline{z}}$ is the cokernel of the induced morphism of logarithmic inertia groups $I^{\text{log}}_{\overline{z}} \to I^{\text{log}}_{\overline{s}}$; this morphism is nothing but $\text{Hom}(\alpha_{\overline{z}},\widehat{\mathbb{Z}}'(1))$, with $\alpha$ as above.

If $h(x) = 1$, then $\mathcal{C}_{\overline{z}} = 0$ and therefore $$\mathrm{Tr}\left(\varphi^d \mid R\psi^t_\mathcal{X}(\Lambda)_{\overline{z}}\right) = \mathrm{Tr}\left(\varphi^d \mid R^0\psi^t_\mathcal{X}(\Lambda)_{\overline{z}}\right).$$ Since moreover $$R^0\psi^t_\mathcal{X}(\Lambda)_{\overline{z}} \cong \Lambda[E_{\overline{z}}],$$ where  $$E_{\overline{z}} = \text{coker}(\text{Hom}(\alpha_{\overline{z}},\widehat{\mathbb{Z}}'(1))$$ has order $s(x)'$, and the tame inertia group $I^t$ acts through  $$I^t \cong I_{\overline{s}}^\text{log} \hookrightarrow \Lambda[I_{\overline{s}}^\text{log}] \longrightarrow \Lambda[E_{\overline{z}}],$$ the trace of $\varphi^d$ on $R^0\psi^t_\mathcal{X}(\Lambda)_{\overline{z}}$ equals $s(x)'$ if $s(x)'$ divides $d$, and to $0$ else.

If $h(x) \geq 2$, then $$n := \mathrm{rank}_\mathbf{Z}\,\mathcal{C}_{\overline{z}} = h(x) - 1 \geq 1.$$ Hence (\ref{wedgeproduct}) yields \begin{eqnarray*} \mathrm{Tr}\left(\varphi^d \mid R\psi^t_\mathcal{X}(\Lambda)_{\overline{z}}\right)  & = & \sum_{i \geq 0} (-1)^i \mathrm{Tr}\left(\varphi^d \mid R^i \psi^t_\mathcal{X}(\Lambda)_{\overline{z}}\right) \\ & = & \sum_{i \geq 0} (-1)^i {n \choose i} \mathrm{Tr}\left(\varphi^d \mid R^0 \psi^t_\mathcal{X}(\Lambda)_{\overline{z}}\right) \\ & = & 0. \end{eqnarray*}

\end{proof}

We continue with a technical lemma on monoids.

\begin{lemm} \label{lemma:monoids} Let $P$ be an fs monoid. Let $\alpha: \mathbf{N} \to P$ be a vertical homomorphism and let $d \in \mathbf{Z}_{> 0}$. Let $Q = P \oplus_{\mathbf{N}} \frac1d \mathbf{N}$. The submonoid $Q^{\mathrm{sat}}$ of $Q^{\mathrm{gp}}$ is generated by $Q^\times$, the $d$-torsion $Q^{\mathrm{gp}}[d]$ and $I \subseteq Q^{\mathrm{sat}}$, the radical of the ideal generated by $P \setminus P^\times$. \end{lemm}

\begin{proof} Let $e = \alpha(1)$. Since the homomorphism $\alpha$ is vertical, $e$ is not contained in any proper face of $P$. This implies that every non-maximal face of $Q$ is of the form $F \oplus \{0\}$, where $F$ is a non-maximal face of $P$. In particular, we have an isomorphism $$P^\times \to Q^\times: p \mapsto (p,0).$$ Take $(p,\frac{m}{d}) \in Q^{\mathrm{sat}}$, where $p \in P^{\mathrm{gp}}$ and $m \in \mathbf{Z}$. There exists an $N \in \mathbf{Z}_{>0}$ such that $N(p,\frac{m}{d}) \in Q$. Multiplying by $d$ if necessary, we can assume that $d \mid N$; hence $$N\left(p,\frac{m}{d}\right) = \left(Np,\frac{Nm}{d}\right) = \left(Np + \frac{Nm}{d}e,0\right),$$ and therefore $$Np +\frac{Nm}{d}e = \frac{N}{d}(dp + me) \in P.$$ Since $P$ is saturated, $dp + me \in P$. Conversely, if $dp + me \in P$, clearly $(p,\frac{m}{d}) \in Q^{\mathrm{sat}}$. 

If $dp + me \in P \setminus P^\times$, then $d(p,\frac{m}{d}) = (dp + me,0)$ lies in the ideal of $Q^{\mathrm{sat}}$ generated by the image of the maximal ideal of $P$; hence $(p,\frac{m}{d})$ lies in $I$, so we are done. 

Let us therefore assume that $dp + me \in P^\times$. Then $(p,\frac{m}{d}) \in (Q^{\mathrm{sat}})^\times$. Let us consider the image $\overline{x}$ of $x = (p,\frac{m}{d})$ in $Q^{\mathrm{sat}}/Q^\times$. It is clear that $\overline{x} \in (Q^{\mathrm{sat}}/Q^\times)^{\times}$. Hence there exists $y \in Q^{\mathrm{sat}}$ such that $\overline{x} + \overline{y} = \overline{0}$ in $Q^{\mathrm{sat}}/Q^\times$, i.e. $x + y \in Q^\times$. Since $dx \in Q$ and $dy \in Q$, we obtain $dx \in Q^\times$. Using the isomorphism $Q^\times \cong P^\times$ mentioned above, we see that there exists $p' \in P^\times$ such that $d(p,\frac{m}{d}) = d(p',0)$. Now $(p - p',\frac{m}{d})$ is clearly $d$-torsion; hence $$\left(p,\frac{m}{d}\right) = (p',0) + \left(p - p',\frac{m}{d}\right)$$ is the sum of an element of $Q^\times$ and an element of $Q^{\mathrm{gp}}[d]$, so we are done. \end{proof}

This leads to the following result:

\begin{prop} \label{prop:degree} Assume that $\mathcal{X}^\dagger$ (as above) is log regular. Fix $x \in F(\mathcal{X})$ and let $d \in \mathbf{Z}_{>0}$ be prime to $p$. The reduced inverse image of $U(x)$ in $\mathcal{X}(d)^\dagger$ is a finite \'etale cover of $U(x)$. The degree of this cover divides $d$ and $\mathrm{gcd}_{y \in x^{(1)}} s(y)$, where $x^{(1)}$ is the set of generizations of $x$ in $F(\mathcal{X})^{(1)}$.
\end{prop}

\begin{proof} On an affine neighbourhood $\mathrm{Spec}\,A$ around $x$, we have an fs chart for the morphism $\mathcal{X}^\dagger \to R^\dagger$ given by a commutative diagram of the form 
\begin{center}
\begin{tikzpicture}[auto]
\node (L1) {$\mathbf{N}$};
\node (L2) [below= 1.2cm of L1] {$P$};
\node (M1) [right= 1.2cm  of L1] {$R$};
\node (M2) at (M1 |- L2) {$A$};
\draw[->] (L1) to node {\footnotesize $1 \mapsto \pi$} (M1);
\draw[->] (L2) to node {} (M2);
\draw[->] (L1) to node  {} (L2);
\draw[->] (M1) to node {} (M2);
\end{tikzpicture} \end{center}
Similarly, a chart for the log structure on $\mathcal{X}(d)^\dagger$ is given by the diagram \begin{center}
\begin{tikzpicture}[auto]
\node (L1) {$\mathbf{N}$};
\node (L2) [below= 1.2cm of L1] {$Q^{\mathrm{sat}}$};
\node (M1) [right= 1.2cm  of L1] {$R(d)$};
\node (M2) at (M1 |- L2) {$B$};
\draw[->] (L1) to node {} (M1);
\draw[->] (L2) to node {} (M2);
\draw[->] (L1) to node  {} (L2);
\draw[->] (M1) to node {} (M2);
\end{tikzpicture} \end{center}
where $$Q = P \oplus_{\mathbf{N}} \frac{1}{d}\mathbf{N} \ \text{ and }\ B = A \otimes_{R[P]} R[Q^\mathrm{sat}].$$
Now $U(x) \cap \mathrm{Spec}\,A$ is nothing but $$\mathrm{Spec}\,A  \times_{\mathrm{Spec}\,R[P]} \mathrm{Spec}\left(R[P]/(\mathfrak{p})\right)_{\mathfrak{p}},$$ where $\mathfrak{p}$ is the prime ideal of $P$ which corresponds to $x$.  Here $(\mathfrak{p})$ denotes the ideal of $R[P]$ generated by $\mathfrak{p}$, and $_\mathfrak{p}$ denotes localization with respect to the multiplicative subset of $R[P]/(\mathfrak{p})$ generated by $P \setminus \mathfrak{p}$. We may and will assume that $\mathfrak{p}$ is the maximal ideal of $P$. Hence we can identify $U(x) \cap \mathrm{Spec}\,A$ with the spectrum of $A \otimes_{\mathbf{Z}[P]} \mathbf{Z}[P^\times]$. The (reduced) inverse image of $U(x)$ in $\mathcal{X}(d)^\dagger$ can be described in a similar way. The fact that this inverse image is \'etale over $U(x)$ is now an immediate consequence of Lemma \ref{lemma:monoids} (recall that we have chosen $d$ prime to $p$).

We will now analyze the degree of this \'etale cover. We can safely assume that $P^\mathrm{gp}$ is torsion free. Choose an isomorphism $P^{\mathrm{gp}} \cong \mathbf{Z}^r$, for some $r > 0$. Let $v$ be the image of $1$ under the composition $\mathbf{N} \to P \to P^\mathrm{gp} \to \mathbf{Z}^r$. Hence we get an identification $$Q^{\mathrm{gp}} \cong (\mathbf{Z}^r \oplus \mathbf{Z})/\langle (v,-d) \rangle.$$ Write $v = \lambda v_0$, where $\lambda \in \mathbf{Z}$ and $v_0$ is a \emph{primitive vector}. Hence the order of the torsion in $Q^{\mathrm{gp}}$ divides $\mathrm{gcd}(d,\lambda)$. The statement about the degree of the \'etale cover now follows from Lemma \ref{lemma:monoids}, together with the observation that $\lambda$ divides the saturation index $s(y)$ for any $y \in x^{(1)}$. Indeed, if $\mathfrak{p}$ is the height $1$ prime ideal of $P$ corresponding to $y$, then $s(y)$ can be identified with the image of $1$ under $P \to P/(P \setminus \mathfrak{p}) \cong \mathbf{N}$. Consider the homomorphisms $$\mathbf{N} \to P \hookrightarrow P^{\mathrm{gp}} \to P^\mathrm{gp}/(P \setminus \mathfrak{p})^{\mathrm{gp}} \cong \mathbf{Z};$$ since the image of $1$ in $P^{\mathrm{gp}}$ is divisible by $\lambda$, so is the integer $s(y)$, since this is simply the image of $1$ in the quotient $P^{\mathrm{gp}}/(P \setminus \mathfrak{p})^{\mathrm{gp}} \cong \mathbf{Z}$. \end{proof}

We can now prove the following result about sheaves of tame nearby cycles, generalizing the results of \cite[\S 2.5]{nic2}:

\begin{theo} \label{lissetamelyramified} As above, assume that $\mathcal{X}^\dagger$ is log regular and fix $x \in F(\mathcal{X})$. Write $\Lambda = \mathbf{Q}_\ell$ for some prime $\ell \neq p$. The sheaves $R^n\psi^t_\mathcal{X}(\Lambda)$ are lisse on $U(x)$ and tamely ramified along the boundary components of any normal compactification. They become constant on a finite \'etale cover of degree dividing $\mathrm{gcd}_{y \in x^{(1)}} s(y)'$.
\end{theo}

\begin{proof} Let $\tilde{s}$ be a log geometric point lying above $s$. Write $\mathcal{X}_{\tilde{s}} = \mathcal{X}_s \times^{\text{fs}}_s \tilde{s}$. We know by \cite[Corollary 8.4]{illusie} that the tame nearby cycles complex is given by the formula $$R\psi^t_{\mathcal{X}}(\Lambda) = R\tilde{\varepsilon}_{\star}(\Lambda|_{\mathcal{X}_{\tilde{s}}}),$$ where $\tilde{\varepsilon}$ is the composition $$\mathcal{X}_{\tilde{s}}^{\text{ket}} \stackrel{\varepsilon}{\longrightarrow} \mathcal{X}_{\tilde{s}}^{\text{et}} \stackrel{\alpha}{\longrightarrow} \mathcal{X}_{s}^{\text{et}}.$$ Here $\varepsilon$ is the canonical map from the Kummer \'etale site to the classical \'etale site; of course $\alpha$ need not be an isomorphism at the level of the underlying schemes. Hence $$R^0\psi^t_{\mathcal{X}}(\Lambda) = \alpha_\star(\Lambda|_{X_{\tilde{s}}}).$$ By \cite[Proposition 1.3.4.1]{vidal}, $\alpha$ is the composition of a surjective closed immersion $\beta_1$ and a finite morphism $\beta_2$. More precisely, $\alpha$ factors as $$\mathcal{X}_{\tilde{s}} \stackrel{\beta_1}{\longrightarrow} \mathcal{X}_{\tilde{s}_n} \stackrel{\beta_2}{\longrightarrow} \mathcal{X}_{s}.$$ Here $\mathcal{X}_{\tilde{s}_n}$ denotes the log scheme $\mathcal{X}_s \times_{s}^{\text{fs}} \tilde{s}_n$, which is nothing but the special fibre of $\mathcal{X}(n)^\dagger$; the integer $n$ is the biggest prime-to-$p$ divisor of the saturation index of $\mathcal{X}$ over $R$. The degree $d_x = \sum_{\alpha(y) = x} [\kappa(y) : \kappa(x)]$ of $\beta_2$ in a point $x$ equals the cardinality of $E_{\overline{x}} = \text{coker}\,\text{Hom}(\varphi_{\overline{x}},\widehat{\mathbf{Z}}'(1))$, where $\varphi$ is the induced map $f^\star(\mathcal{M}^{\sharp}_s)^{\text{gp}} \longrightarrow (\mathcal{M}_{\mathcal{X}_s}^\sharp)^\text{gp}$. 

Let us now fix $x \in F(\mathcal{X})$. We know that $$R^0\psi^t_{\mathcal{X}}(\Lambda) = \alpha_\star(\Lambda|_{X_{\tilde{s}}}) = (\beta_2)_{\star}(\Lambda|_{X_{\tilde{s}_n}}),$$ hence the restriction of $R^0\psi^t_{\mathcal{X}}(\Lambda)$ to $U(x)$ becomes constant on the inverse image of $U(x)$ in $\mathcal{X}(n)$, which is \'etale by Proposition \ref{prop:degree}; the statement about the degree is a consequence of the fact that $n$ is the biggest prime-to-$p$ divisor of $\mathrm{lcm}_{y \in x^{(1)}} s(y)$.

Hence $R^n \psi_{\mathcal{X}}(\Lambda)$ becomes constant on the same cover, for any $n \geq 0$, since the sheaf $\mathcal{C}$ in the proof of Lemma \ref{localcomp} is constant on $U(x)$. It follows that the sheaves $R^n \psi^t_{\mathcal{X}}(\Lambda)$ are lisse, and tamely ramified along the boundary of any normal compactification. \end{proof}

We are now ready to calculate the trace of the tame monodromy operator:

\begin{theo} \label{theoremtrace} Let $\mathcal{X}$ be a proper, flat $R$-model such that $\mathcal{X}^\dagger$ is log regular. Then \begin{equation}   \mathrm{Tr}\left(\varphi^d \mid H^\star(X \times_K K^t,\mathbf{Q}_\ell)\right) = \sum_{\substack{x \in F(\mathcal{X})^{(1)} \\ s(x)' \mid d\ }} s(x)' \chi(U(x)). \label{trace} \end{equation}  for every positive integer $d$.
\end{theo}
\begin{proof} We will prove a more general result: if $Z$ is a subscheme of $\mathcal{X}_s$, then \begin{equation} \sum_{m \geq 0} (-1)^m \mathrm{Tr} \left(\varphi^d \mid \mathbf{H}^m_c(Z,R\psi^t_\mathcal{X}(\mathbf{Q}_\ell)|_Z)\right) = \sum_{\substack{x \in F(\mathcal{X})^{(1)} \\ s(x)' \mid d\ }} s(x)' \chi(U(x) \cap Z). \label{traceZ} \end{equation} In particular, (\ref{trace}) follows from (\ref{traceZ}) by taking $Z = \mathcal{X}_s$ and applying the spectral sequence for nearby cycles, since $\mathcal{X}$ is assumed to be proper over $S$. To prove the equality (\ref{traceZ}), note that both sides are additive with respect to partitioning $Z$ into subvarieties; this allows us to assume that $Z$ is contained in $U(x)$, for some $x \in F(\mathcal{X})$. By the spectral sequence for hypercohomology, the left hand side of (\ref{traceZ}) equals $$\sum_{p, q \geq 0} (-1)^{p + q} \mathrm{Tr}\left(\varphi^d \mid H^p_c(Z,R^q\psi^t_{\mathcal{X}}(\mathbf{Q}_\ell)|_Z)\right).$$ We can assume that $Z$ is normal and choose a normal compactification $Z^\text{c}$. We know (Proposition \ref{lissetamelyramified}) that the sheaves $R^n\psi^t_{\mathcal{X}}(\mathbf{Q}_\ell)|_Z$ are lisse on $Z$, and tamely ramified along each of the irreducible components of $Z^c \setminus Z$. Let $z \in Z$ and let $\overline{z}$ be a geometric point lying over $z$. Using Nakayama's description of the action of the monodromy operator on the stalks $R^n\psi^t_\mathcal{X}(\mathbf{Q}_\ell)_{\overline{z}}$ -- cfr. the proof of Lemma \ref{localcomp} -- we see that this action has finite order. Hence we can apply \cite[Lemma 5.1]{NicaiseSebag}, which gives $$\sum_{p \geq 0} (-1)^p \mathrm{Tr}\left(\varphi^d \mid H^p_c(Z,R^q\psi^t_{\mathcal{X}}(\mathbf{Q}_\ell)|_Z)\right) = \chi(Z) \cdot \mathrm{Tr}(\varphi^d \mid R^q \psi^t_{\mathcal{X}}(\mathbf{Q}_{\ell})_{\overline{z}}).$$ Using Lemma (\ref{localcomp}), we see that the only contributions to the left-hand side of the equality (\ref{traceZ}) will come from the height $1$ points in the fan $F(\mathcal{X})$: we get \begin{align*} & \sum_{m \geq 0} (-1)^m \mathrm{Tr} \left(\varphi^d \mid \mathbf{H}^m_c(Z,R\psi^t_\mathcal{X}(\mathbf{Q}_\ell)|_Z)\right) \\  = \ \ & \chi(Z) \cdot \sum_{q \geq 0} (-1)^q \mathrm{Tr}(\varphi^d \mid  R^q \psi^t_{\mathcal{X}}(\mathbf{Q}_{\ell})_{\overline{z}}) \\ = \ \  & s(x)' \chi(Z)\end{align*}  if the point $x \in F(\mathcal{X})^{(1)}$ satisfies $s(x)' \mid d$; in all other cases, we get no contribution. This concludes the proof of the theorem. \end{proof}

As a corollary, we get a formula for the tame monodromy zeta function of $X$: \begin{eqnarray} \zeta_X(t) & = & \det\left(t \cdot \mathrm{Id} - \varphi \mid H^{\star}(X \times_K K^t,\mathbf{Q}_\ell)\right) \nonumber \\ & = & \prod_{m \geq 0} \det\left(t \cdot \mathrm{Id} - \varphi \mid H^m(X \times_K K^t,\mathbf{Q}_\ell)\right)^{(-1)^{m + 1}}.\end{eqnarray} The result is an A'Campo type formula \cite[\S 1]{A'Campo}, generalizing the results of \cite[\S 2.6]{nic2}:

\begin{coro} We have \begin{equation} \label{A'Campo} \zeta_X(t) \ = \  \prod_{\substack{x \in F(\mathcal{X})^{(1)} \\ s(x)' \mid d\ }} (t^{s(x)'} - 1)^{-\chi(U(x))}.\end{equation} \end{coro} 
\begin{proof} As in the proof of Theorem \ref{theoremtrace}, we can actually prove a more general result: if $Z$ is any subscheme of $\mathcal{X}_s$, then the equality \begin{equation} \det\left(t \cdot \mathrm{Id} - \varphi \mid \mathbf{H}^\star_{\text{c}} (Z,R\psi^t_{\mathcal{X}}(\mathbf{Q}_\ell)|_Z)\right) = \prod_{\substack{x \in F(\mathcal{X})^{(1)} \\ s(x)' \mid d\ }} (t^{s(x)'} - 1)^{-\chi(U(x) \cap Z)} \label{ACampoZ} \end{equation} holds. This formula can easily be deduced from the equality (\ref{traceZ}) using the fact that, for every field $F$ of characteristic zero, every finite dimensional $F$-vector space $V$ and every endomorphism $g$ of $V$, we have the classical identity $$\det(\mathrm{Id} - t \cdot g \mid V)^{-1} = \exp\left(\sum_{d > 0} \mathrm{Tr}(g^d \mid V) \frac{t^d}{d}\right)$$ in $F[[t]]$. (For similar arguments, we refer to \cite[1.5]{weil1} and \cite[\S 1]{A'Campo}.) \end{proof}

\section{Log blow-ups and the rational volume} \label{sec:rationalvolume}

As in the previous section, we assume that we are given a smooth, proper $K$-variety $X$ for which there exists a proper, flat $R$-model $\mathcal{X}$ such that $\mathcal{X}^\dagger$ is log regular. We will now compute the rational volume of $X$ in terms of the logarithmic stratification. 

We start with a simple lemma.

\begin{lemm} \label{verticallogstructure} Under the above assumptions, the generic fibre $X = \mathcal{X} \times_R K$ is smooth. Moreover, the following statements are equivalent:
\begin{itemize}
\item[(A)] the scheme $\mathcal{X}$ is regular and $\mathcal{X}_s$ is a divisor with strict normal crossings;
\item[(B)] for every point $x \in F(\mathcal{X})$, we have $\mathcal{M}_{F(\mathcal{X}),x} \cong \mathcal{M}_{\mathcal{X},x}^\sharp \cong \mathbf{N}^{h(x)}$.
\end{itemize}
As before, $h(x)$ denotes the height of the point $x$ in the fan $F(\mathcal{X})$.
\end{lemm} 
\begin{proof} The fact that the generic fibre is smooth follows from the verticality of the log structure. The underlying scheme $\mathcal{X}$ is regular if and only if for every $x \in F(\mathcal{X})$, the monoid $\mathcal{M}_{F(\mathcal{X}),x}$ is free and finitely generated by \cite[Corollary 9.5.35]{GR}. 

It remains to check that the statement about the special fibre. Assume that (B) holds and fix any $x \in F(\mathcal{X})$. Locally around $x$, we have a chart given by 
\begin{center}
\begin{tikzpicture}[auto]
\node (L1) {$\mathbf{N}$};
\node (L2) [below= 1.2cm of L1] {$R$};
\node (M1) [right= 1.2cm  of L1] {$\mathcal{O}_{\mathcal{X},x}^\times \oplus  \mathbf{N}^{h(x)}$};
\node (M2) at (M1 |- L2) {$\mathcal{O}_{\mathcal{X},x}$};
\draw[->] (L1) to node {} (M1);
\draw[->] (L2) to node {} (M2);
\draw[->] (L1) to node  {} (L2);
\draw[->] (M1) to node {\footnotesize $e_i \mapsto x_i$} (M2) {};
\end{tikzpicture} \end{center}
 
where $(e_i)_{1 \leq i \leq h(x)}$ is the standard basis for $\mathbf{N}^{h(x)}$ and \begin{equation} \label{sncd} \pi = u \prod_{i = 1}^{h(x)} x_i^{n_i}\ \text{ in } \mathcal{O}_{\mathcal{X},x} \end{equation} for some $u \in \mathcal{O}_{\mathcal{X},x}^\times$ and positive integers $(n_i)_{1 \leq i \leq h(x)}$. Since $I(x,\mathcal{M}_\mathcal{X}) = (x_1,\cdots,x_{h(x)})$ and $C_{X,x} = \mathcal{O}_{\mathcal{X},x}/I(x,\mathcal{M}_{\mathcal{X}})$ is a regular local ring by \cite[Definition 2.1]{kato}, we see that $\mathcal{X}$ has strict normal crossings at $x$ by \cite[Proposition 4.2.15]{liu}. 
\end{proof}

In the situation of Proposition \ref{verticallogstructure}, we have \begin{equation} \label{sncdformula} \mathcal{X}_s = \sum_{x \in F(\mathcal{X})^{(1)}} s(x) \, \overline{U(x)}. \end{equation} 
Let $Y$ be a smooth, projective variety over $K$. If there exists an \emph{sncd model} $\mathcal{Y}$ of $Y$ as in Proposition \ref{verticallogstructure}, then the rational volume can be read off from the formula (\ref{sncdformula}) following \cite[Proposition 4.2.1]{nic2}: we have \begin{equation} \label{rationalvolume} \mathrm{s}(Y) = \sum_{\substack{y \in F(\mathcal{Y})^{(1)} \\ s(y) = 1\ }}  \chi(U(y)). \end{equation} Indeed, $\mathcal{Y}$ is a regular and proper model of $Y$; its smooth locus (over $R$) is a weak N\'eron model for $Y$. The smooth locus of $\mathcal{Y}_s$ is precisely the open subscheme $$\bigsqcup_{y \in F(\mathcal{Y})^{(1)},\ s(y) = 1} U(y),$$ whence equality (\ref{rationalvolume}) follows. 

We can now state and prove the main result of this section:

\begin{theo} \label{sX} Let $\mathcal{X}$ be a proper, flat $R$-scheme such that $\mathcal{X}^\dagger$ is log regular. Let $X = \mathcal{X} \times_R K$ be the generic fibre. Then \begin{equation} \mathrm{s}(X) = \sum_{\substack{x \in F(\mathcal{X})^{(1)} \\ s(x) = 1\ }}  \chi(U(x)).\label{rationalvolumeexpression} \end{equation} \end{theo}

The idea is simple: we consider a subdivision $\varphi: F' \to F(\mathcal{X})$ of the fan in the sense of \cite[Definition 9.6]{kato} such that $\varphi^\star \mathcal{X}$ satisfies the equivalent conditions of Proposition \ref{verticallogstructure}; in particular, such that the underlying scheme of $\varphi^\star \mathcal{X}$ is (classically) regular. This is possible by \cite[\S 10.4]{kato}. We can then compute the rational volume starting from the modified model $\varphi^\star \mathcal{X}$ using the formula (\ref{rationalvolume}). The key technical ingredient needed for this computation is the fact that the log blow-up $\mathrm{Bl}_{\varphi}: \varphi^\star \mathcal{X} \to \mathcal{X}$ is a piecewise trivial fibration in tori; this is the content of the following lemma.

\begin{lemm} \label{keylemma} Let $\mathcal{X}$ be a flat $R$-scheme such that $\mathcal{X}^\dagger$ is log regular. Consider a subdivision $\varphi: F' \to F(\mathcal{X})$ of its fan. Fix $x' \in F'$ and $x \in \mathcal{X}$ such that $\varphi(x') = \pi(x)$, i.e. $x \in U(\varphi(x')) \subseteq \mathcal{X}$. Then we have $$U(x') \cap \emph{Bl}_\varphi^{-1}(x) \cong \mathbf{G}^{h(\varphi(x')) - h(x')}_{m,\kappa(x)}$$ where $\emph{Bl}_{\varphi}^{-1}(x)$ is the fibre of the log blow-up $\varphi^{\star} \mathcal{X} \to \mathcal{X}$ above $x$. \label{blowups} \end{lemm}

\begin{proof} Fix $x' \in F'$ and $x \in \mathcal{X}$ such that $x \in U(\varphi(x'))$. The statement is local on $\mathcal{X}$. Hence, by shrinking $\mathcal{X}$ if necessary, we may assume that $F(\mathcal{X}) = \mathrm{Spec}\,P$ and $F' = F(\varphi^\star \mathcal{X}) = \mathrm{Spec}\,P'$, where $P$ and $P'$ are sharp fs monoids. 

We can also assume that the natural morphism $P \to \mathcal{M}_\mathcal{X}^\sharp$ lifts to a homomorphism of sheaves $P \to \mathcal{M}_X$ such that the composition $P \to \mathcal{M}_X \to \mathcal{O}_X$ gives a chart for the log structure on $\mathcal{X}$. Define $Q$ by $$Q = P^{\mathrm{gp}} \times_{(P')^{\mathrm{gp}}} P';$$ this is the submonoid of $P^\mathrm{gp}$ consisting of the elements whose image in $(P')^{\mathrm{gp}}$ lies inside $P'$. The obvious map $P^{\text{gp}} \to (P')^{\text{gp}}$ is surjective; $Q$ does not need to be sharp, but by \cite[Lemma 9.6.12]{GR} one has $Q^\sharp \cong (P')^\sharp$. The log blow-up $\varphi^\star \mathcal{X}$ is given by $\varphi^\star \mathcal{X} = \mathcal{X} \times_{\mathrm{Spec}\,R[P]} \mathrm{Spec}\,R[Q]$, cfr. the construction in \cite[Proposition 9.6.14]{GR}. The log structure on $\varphi^\star \mathcal{X}$ is given by the chart $Q \to \mathcal{O}_{\varphi^\star \mathcal{X}}$.  

Let $\psi: P \to Q$ be the natural injection, let $\mathfrak{q}$ be the prime ideal of $Q$ corresponding to $x'$ and let $\mathfrak{p} = \psi^{-1}(\mathfrak{q})$ be the prime ideal of $P$ corresponding to $\varphi(x')$. The stratum $U(\varphi(x'))$ can then be desribed as $$\mathcal{X} \times_{\mathrm{Spec}\,R[P]} \mathrm{Spec}\left(R[P]/(\mathfrak{p})\right)_{\mathfrak{p}},$$ and its Zariski closure $\overline{U(\varphi(x'))}$ is simply $$\mathcal{X} \times_{\mathrm{Spec}\,R[P]} \mathrm{Spec}\,R[P]/(\mathfrak{p});$$ here $(\mathfrak{p})$ denotes the ideal of $R[P]$ generated by the elements of $\mathfrak{p}$, and the subscript $_\mathfrak{p}$ denotes localization with respect to the multiplicative subset of $R[P]/(\mathfrak{p})$ generated by the image of $P \setminus \mathfrak{p}$. Similarly, the stratum $U(x')$ is now equal to $$\varphi^\star \mathcal{X} \times_{\mathrm{Spec}\,R[Q]} \mathrm{Spec}\left(R[Q]/(\mathfrak{q})\right)_{\mathfrak{q}}.$$ 

We have the following pullback diagram:
\begin{center}
\begin{tikzpicture}[auto]
\node (L1) {$U(x')$};
\node (L2) [below= 1.3cm of L1] {$\mathrm{Spec}\left(R[Q]/(\mathfrak{q})\right)_{\mathfrak{q}}$};
\node (M1) [right= 2.5cm  of L1] {$U(\varphi(x'))$};
\node (M2) at (M1 |- L2) {$\mathrm{Spec}\left(R[P]/(\mathfrak{p})\right)_{\mathfrak{p}}$};
\draw[->] (L1) to node {} (M1);
\draw[->] (L2) to node {} (M2);
\draw[->] (L1) to node  {} (L2);
\draw[->] (M1) to node {} (M2);
\end{tikzpicture} \end{center}
and hence, taking fibres over $x \in U(\varphi(x'))$, the pullback diagram 
  
  \begin{center}
\begin{tikzpicture}[auto]
\node (L1) {$U(x') \cap \text{Bl}_\varphi^{-1}(x)$};
\node (L2) [below= 1.3cm of L1] {$\mathrm{Spec}\left(\kappa(x)[Q]/(\mathfrak{q})\right)_{\mathfrak{q}}$};
\node (M1) [right= 2.5cm  of L1] {$\mathrm{Spec}\,\kappa(x) $};
\node (M2) at (M1 |- L2) {$\mathrm{Spec}\left(\kappa(x)[P]/(\mathfrak{p})\right)_{\mathfrak{p}}$};
\draw[->] (L1) to node {} (M1);
\draw[->] (L2) to node {$\star$} (M2);
\draw[->] (L1) to node  {} (L2);
\draw[->] (M1) to node {} (M2);
\end{tikzpicture} \end{center}
Therefore it now suffices to prove that the fibres of $$\mathrm{Spec}\left(\kappa(x)[Q]/(\mathfrak{q})\right)_{\mathfrak{q}} \to  \mathrm{Spec}\left(\kappa(x)[P]/(\mathfrak{p})\right)_{\mathfrak{p}}$$ are (split) tori of the required rank. Notice that $$\kappa(x)[Q]/(\mathfrak{q}) \cong \kappa(x)[Q \setminus \mathfrak{q}]$$ and hence that $$(\kappa(x)[Q]/(\mathfrak{q}))_{\mathfrak{q}} \cong \kappa(x)[(Q \setminus \mathfrak{q})^{\text{gp}}].$$ Similarly, we obtain that $$\left(\kappa(x)[P]/(\mathfrak{p})\right)_{\mathfrak{p}} \cong \kappa(x)[(P \setminus \mathfrak{p})^{\text{gp}}].$$ 

We have a commutative diagram   
\begin{center}
\begin{tikzpicture}[auto]
\node (L1) {$ \left(\kappa(x)[P]/(\mathfrak{p})\right)_{\mathfrak{p}} $};
\node (L2) [below= 1.3cm of L1] {$ \kappa(x)[(P \setminus \mathfrak{p})^{\text{gp}}]$};
\node (M1) [right= 2.5cm  of L1] {$(\kappa(x)[Q]/(\mathfrak{q}))_{\mathfrak{q}}$};
\node (M2) at (M1 |- L2) {$\kappa(x)[(Q \setminus \mathfrak{q})^{\text{gp}}] $};
\draw[->] (L1) to node {} (M1);
\draw[->] (L2) to node {} (M2);
\draw[->] (L1) to node  {} (L2);
\draw[->] (M1) to node {} (M2);
\end{tikzpicture} \end{center}
 in which the vertical maps are the isomorphisms mentioned above, the top horizontal arrow is the one which induces the map $\star$ in the previous diagram and the bottom horizontal arrow is induced by the injective homomorphism $P \setminus \mathfrak{p} \to Q \setminus \mathfrak{q}$ obtained by restricting $\psi: P \to Q$. 
  
  Now $P \setminus \mathfrak{p}$ and $Q \setminus \mathfrak{q}$, being submonoids of fs monoids, are themselves fs; the quotient $(Q \setminus \mathfrak{q})^{\text{gp}}/(P \setminus \mathfrak{p})^{\text{gp}}$ is a subgroup of $P^{\text{gp}}/(P \setminus \mathfrak{p})^{\text{gp}}$, since $Q$ (and hence $Q \setminus \mathfrak{q}$) is a submonoid of $P$. The latter group is a free abelian group of finite type, hence so is its subgroup $(Q \setminus \mathfrak{q})^{\text{gp}}/(P \setminus \mathfrak{p})^{\text{gp}}$. This proves that the fibres of $\star$ are split tori; it remains to compute their rank. We have the equalities \begin{align*} h(\varphi(x')) - h(x') & = \ \dim Q_\mathfrak{q} - \dim P_\mathfrak{p} \\ & = \  \text{rank}_\mathbf{Z}\,(Q_{\mathfrak{q}}^\sharp)^\text{gp} - \text{rank}_\mathbf{Z}\, (P_{\mathfrak{p}}^\sharp)^{\text{gp}}\end{align*} (the last one uses the fact that $P$ and $Q$ are fine); it is now easy to see that $$\text{rank}_\mathbf{Z}\, (P_{\mathfrak{p}}^\sharp)^{\text{gp}} = \text{rank}_\mathbf{Z}\,(P \setminus \mathfrak{p})^\text{gp}$$ and  $$\text{rank}_\mathbf{Z}\, (Q_{\mathfrak{q}}^\sharp)^{\text{gp}} = \text{rank}_\mathbf{Z}\,(Q \setminus \mathfrak{q})^\text{gp},$$ so we are done.
  \end{proof}
  We are now in the position to prove Theorem \ref{sX}.
\begin{proof}[Proof of Theorem \ref{sX}] Choose a subdivision $\varphi: F' \to F(\mathcal{X})$ such that $$\mathcal{M}_{F',x'} \cong \mathbf{N}^{h(x')}$$ for every $x' \in F'$. Let $\text{Bl}_\varphi: \varphi^\star \mathcal{X} \to \mathcal{X}$ be the log blow-up associated with $\varphi$ -- this is a ``resolution of toric singularities \`a la Kato'' \cite[\S 10.4]{kato}. The log scheme $\varphi^\star \mathcal{X}$ satisfies the equivalent conditions of Proposition \ref{verticallogstructure}; hence (\ref{rationalvolume}) applies to $\varphi^\star \mathcal{X}$. Since the Euler characteristic is additive with respect to partitions into locally closed subsets and the Euler characteristic of a torus is $0$, the result follows from Lemma \ref{blowups}.\end{proof}

\section{Proof of the main theorem}

We will now finish the proof of Theorem (\ref{maintheorem}). From (\ref{trace}) and (\ref{rationalvolumeexpression}), one immediately obtains the following expression for the so-called ``error term'' \cite[Definition 6.7]{nic1}, generalizing \cite[Theorem 7.3]{nic1} in the case of curves: we have $$\varepsilon(X) := \sum_{i \geq 0} (-1)^i \mathrm{Tr}(\varphi \mid H^i(X \times_K K^t,\mathbf{Q}_\ell)) - \mathrm{s}(X) = \sum_{\substack{x \in F(\mathcal{X})^{(1)} \\ s(x) = p^r,\ r \geq 1}} \chi(U(x)).$$ This formula is valid for any log regular model $\mathcal{X}$, and $\varepsilon(X)$ does not vanish in general. However, it vanishes if $X$ admits a log smooth model. This is an immediate consequence of the following result, which is interesting in its own right:

\begin{prop} Let $\mathcal{X}$ be a proper $R$-model of $X$ such that $\mathcal{X}^\dagger$ is log smooth over $R^\dagger$. Let $x \in F(\mathcal{X})^{(1)}$ such that $p$ divides $s(x)$. Then $\chi(U(x)) = 0$.\end{prop}

In the case of curves, we know that cohomological tameness implies logarithmic good reduction, by work of Stix \cite[Theorem 1.2]{stix}. Moreover, Saito's criterion for cohomological tameness of curves \cite{saito} gives a precise description of the irreducible components of the special fibre of a log smooth model for which the multiplicity is divisible by the residual characteristic $p$. Each such component is a copy of $\mathbf{P}^1$, which intersects exactly two other components, both of which have multiplicity prime to $p$. The above result should be seen as a partial generalization of this description.

\begin{proof} Choose $x \in F(\mathcal{X})$ and $y \in \mathcal{X}$ such that $y \in \overline{U(x)}$. Choose a chart \emph{\`a la Kato} (Definition 2.2) for the log structure around $y$, i.e. an \'etale neighbourhood $V = \mathrm{Spec}\,A$ of $y$ and a homomorphism of monoids $\mathbf{N} \to P$ such that \begin{center}
\begin{tikzpicture}[auto]
\node (L1) {$P$};
\node (L2) [below= 1.2cm of L1] {$\mathbf{N}$};
\node (M1) [right= 1.2cm  of L1] {$A$};
\node (M2) at (M1 |- L2) {$R$};
\draw[->] (L1) to node {\footnotesize $\varphi$} (M1);
\draw[->] (L2) to node  {\footnotesize $1 \mapsto \pi$} (M2);
\draw[->] (L2) to node  {\footnotesize $u$} (L1);
\draw[->] (M2) to node {} (M1) {};
\end{tikzpicture} \end{center} commutes, where $P$ is an fs monoid and $u^{\mathrm{gp}}: \mathbf{N}^{\mathrm{gp}} \to P^{\mathrm{gp}}$  has the property that its kernel and the torsion part of its cokernel $C$ are finite groups, the order of which is invertible on $R$. We can safely assume that $P$ is toric (i.e. $P^{\mathrm{gp}}$ is torsion free).

The point $x \in F(\mathcal{X})$ corresponds to a prime ideal $\mathfrak{p} \subseteq P$ and $\overline{U(x)} \cap V$, seen as a reduced, closed subscheme of $V$, can be described as the fibre product $$V \times_{\mathrm{Spec}\,R[P]} \mathrm{Spec}\,R[P \setminus \mathfrak{p}] \cong \mathrm{Spec}\,A/(\mathfrak{p}),$$ where $(\mathfrak{p})$ is the ideal of $A$ generated by $\{\varphi(p) \mid p \in \mathfrak{p}\}$. Now $\overline{U(x)}$ becomes a log scheme for the log structure defined locally by $P \setminus \mathfrak{p} \to A/(\mathfrak{p})$, and this log scheme is log regular by \cite[Proposition 7.2]{kato}. Since $k$ is perfect, it is even log smooth over $k$ (equipped with the trivial log structure). The log structure on $\overline{U(x)}$ is the one induced by the reduced Weil divisor $\Delta$ supported on the union of the strata $\overline{U(x')}$, where $x'$ is a (strict) specialization of $x$ in the monoidal space $F(\mathcal{X})$. 

From now on, assume that $x \in F(\mathcal{X})^{(1)}$ and that $p$ divides $m = s(x)$. In the local ring $\mathcal{O}_{\mathcal{X},x}$, we have $\pi = uf^m$ where $u$ is a unit and $f$ is a local equation of the irreducible component $\overline{U(x)}$ of $\mathcal{X}_s$. We will study the meromorphic one-form $\mathrm{dlog}\,u$ on $\overline{U(x)}$. If $\pi = vg^m$ is another such factorization in $\mathcal{O}_{\mathcal{X},x}$, we have $\mathrm{dlog}\,u = \mathrm{dlog}\,v$ on $\overline{U(x)}$ since the residue field $k$ has characteristic $p$ and $p$ divides $m$. Hence $\mathrm{dlog}\,u$ is a well-defined meromorphic $1$-form on the irreducible component $\overline{U(x)}$. Locally around $y$, it can be described in terms of the log structure, as follows.

\begin{center}
\begin{tikzpicture}[auto]
\node (L2) {$(P \setminus \mathfrak{p})^{\mathrm{gp}}$};
\node (M2) [right= 1.2 cm of L2] {$P^{\mathrm{gp}}$};
\node (M1) [above= 1.2cm  of M2] {$C$};
\node (M3) [below= 1.2cm of M2] {$\mathbf{N}^{\mathrm{gp}}$};
\node (R2) [right= 1.2cm of M2] {$\mathbf{N}^{\mathrm{gp}}$};
\draw[->] (L2) to node {} (M2);
\draw[->>] (M2) to node {\footnotesize $r$} (R2);
\draw[->] (M3) to node {\footnotesize $u^{\mathrm{gp}}$} (M2);
\draw[->>] (M2) to node {} (M1);

\end{tikzpicture} \end{center} 

Let $a = u(1) \in P$, then $\varphi(a) = \pi$ and $r(a) = m \in \mathbf{N}$. Choose $b \in P$ such that $r(b) = 1$, i.e. $\varphi(b)$ is a uniformizer in the discrete valuation ring $\mathcal{O}_{\mathcal{X},x}$. We have $c := mb - a \in (P \setminus \mathfrak{p})^{\mathrm{gp}}$. Taking $u = \varphi(c) \in \mathcal{O}_{\mathcal{X},x}^\times$ and $f = \varphi(b)$, we have $\pi = uf^m$.

The sheaf $\Omega^1_{\overline{U(x)}}(\mathrm{log}\,\Delta)$ is locally free (by log regularity) and $$\left.\Omega^1_{\overline{U(x)}}(\mathrm{log}\,\Delta)\right|_V \cong (P \setminus \mathfrak{p})^{\mathrm{gp}} \otimes \mathcal{O}_{\overline{U(x)}\, \cap \,V}.$$ Now $\mathrm{dlog}\,u$ (when restricted to $\overline{U(x)}$) is a section of $\Omega^1_{\overline{U(x)}}(\mathrm{log}\,\Delta)$ and under the above isomorphism, it corresponds to $c \otimes 1$. This section does not vanish at any point of $\overline{U(x)}$: for that to happen, $c$ would need to be divisible by $p$ in the free abelian group (of finite type) $(P \setminus \mathfrak{p})^{\mathrm{gp}}$. Hence $a = mb - c$ would be divisible by $p$. This means that the cokernel $C$ of $u^{\mathrm{gp}}$ has $p$-torsion, contradicting our assumptions.

The conclusion is that the $1$-form $\mathrm{dlog}\,u$ is a nowhere vanishing (holomorphic) section of the vector bundle  $\Omega^1_{\overline{U(x)}}(\mathrm{log}\,\Delta)$ on the log regular scheme $\overline{U(x)}$. Hence the top Chern class of this vector bundle vanishes. By \cite[``Lemma 0'']{saitooud}, which we state below for the reader's convenience, this implies the result, since $U(x) = \overline{U(x)} \setminus \Delta$. \end{proof}

\begin{lemm}  Let $k$ be an algebraically closed field and let $(X,\mathcal{M}_X)$ be a log regular scheme over $k$ (where $k$ has the trivial log structure). Let $U$ be the locus of triviality of the log structure on $X$.  Then $$\chi(U) = \mathrm{deg}\,c_{X,\mathcal{M}_{X}},$$ where $c_{X,\mathcal{M}_X}$ is the top Chern class of the vector bundle $\Omega^1_X(\mathrm{log}\,\mathcal{M}_X)$ on $X$.
\end{lemm}
To prove this formula, one can use log blow-ups to reduce the statement to the case of a smooth $k$-variety equipped with a divisorial log structure, induced by a strict normal crossings divisor. In that setting, the formula is rather well-known and can be checked by an explicit computation.


\begin{thebibliography}{20}
\footnotesize

\bibitem{A'Campo}
N. A'Campo, \emph{La fonction z\^eta d'une monodromie}. \\ Commentarii Mathematici Helvetici 50 (1975), 233--248.
\bibitem{BS}
A. Bellardini, A. Smeets, \emph{Logarithmic good reduction of abelian varieties}. \\ Preprint (December 2015), 	arXiv:1512.02464.
\bibitem{blr}
S. Bosch, W. L\"utkebohmert, M. Raynaud, \emph{N\'eron models}. \\ 
Ergebnisse der Mathematik und ihrer Grenzgebiete (3) 21, Springer (1990).
\bibitem{weil1}
P. Deligne, \emph{La conjecture de Weil. I.} \\ Publications math\'ematiques de l'IH\'ES 43 (1973), 273--307.

\bibitem{denefloeser}
J. Denef, F. Loeser, \emph{Lefschetz numbers of iterates of the monodromy and truncated arcs}. \\
Topology 41 (2002), 1031--1040.
\bibitem{EN}
H. Esnault, J. Nicaise, \emph{Finite group actions, rational fixed points and weak N\'eron models}. \\ Pure and Applied Mathematics Quarterly 7 (2011), 1209--1240.
\bibitem{fulton}
W. Fulton, \emph{Intersection theory}. \\ Ergebnisse der Mathematik und ihrer Grenzgebiete, Springer Verlag (1984).
\bibitem{GR}
O. Gabber, L. Ramero, \emph{Foundations for almost ring theory (release 6.5)}. \\
Book in preparation (2013).
\bibitem{halnic}
L. H. Halle, J. Nicaise, \emph{N\'eron models and base change}. \\ Lecture Notes in Mathematics 2156, Springer (2016).
\bibitem{illusie}
L. Illusie, \emph{An overview of the work of Fujiwara, Kato and Nakayama on log \'etale cohomology}. \\ Cohomologies $p$-adiques et applications arithm\'etiques II. Ast\'erisque 279 (2002), 271--322.
\bibitem{katof}
K. Kato, \emph{Logarithmic structures of {F}ontaine--{I}llusie}. \\
Algebraic analysis, geometry \& number theory. Johns Hopkins University (1989), 191--224.
\bibitem{kato}
K. Kato, \emph{Toric singularities}. \\
American Journal of Mathematics 116 (1994), 1073--1099.
\bibitem{liu}
Q. Liu, \emph{Algebraic geometry and arithmetic curves}. \\
Oxford Graduate Texts in Mathematics 6 (2002), Oxford University Press. 
\bibitem{LS}
F. Loeser, J. Sebag, \emph{Motivic integration on smooth rigid varieties}. \\ Duke Mathematical Journal 119 (2003), 315--344.
\bibitem{nakayama}
C. Nakayama, \emph{Nearby cycles for log smooth families}. \\ Compositio Mathematica 112 (1998), 45--75.
\bibitem{nic1}
J. Nicaise, \emph{A trace formula for varieties over a discretely valued field}. \\ 
Journal f\"ur die reine und angewandte {Mathematik} 650 (2011), 193--238.
\bibitem{nic2}
J. Nicaise, \emph{Geometric criteria for tame ramification}. \\ 
Mathematische Zeitschrift 273 (2013), 839--868.
\bibitem{NicaiseSebag}
J. Nicaise, J. Sebag, \emph{The motivic Serre invariant, ramification and the analytic Milnor fibre}. \\ Inventiones Mathematicae 168 (2007),
133--173.
\bibitem{niziol}
W. Niziol, \emph{Toric singularities: log blow-ups and global resolutions}. \\ Journal of Algebraic Geometry 15 (2006), 1--29.
\bibitem{ogus}
A. Ogus, \emph{Lectures on logarithmic algebraic geometry}. \\
Book in preparation (2014).
\bibitem{saito}
T. Saito, \emph{Vanishing cycles and geometry of curves over a discrete valuation ring}. \\ American Journal of Mathematics 109 (1987), 1043--1085.
\bibitem{saitooud}
T. Saito, \emph{The Euler numbers of $\ell$-adic sheaves of rank $1$ in positive characteristic}. \\ Proceedings of an ICM 1990 satellite conference (1991), 165--181  (available on  Saito's website).
\bibitem{saito2}
T. Saito, \emph{Log smooth extension of a family of curves and semistable reduction.} \\ Journal of Algebraic Geometry 13 (2004), 287--321.
\bibitem{stix}
J. Stix, \emph{A logarithmic view towards semistable reduction}. \\
Journal of Algebraic Geometry 14 (2005), 119--136.
\bibitem{vidal}
I. Vidal, \emph{Monodromie locale et fonctions z\^eta des log sch\'emas}. \\ Geometric aspects of Dwork theory II. De Gruyter (2004), 983--1039.
\end{thebibliography}
\end{document}